\documentclass[11pt, a4paper,reqno]{amsart}
\usepackage[a4paper, margin=4cm]{}
\usepackage{amssymb,amsfonts, amsthm, xfrac, amscd}
\usepackage{setspace}
\usepackage{amsmath}

\usepackage{color}
\usepackage{array}

\usepackage{color}
\usepackage[bookmarks]{hyperref}

\allowdisplaybreaks[4]

\newtheoremstyle{myremark}     {10pt}{10pt}{}{}{\bfseries}{.}{.5em}{}

\newtheorem{thm}{Theorem}[section]
\newtheorem{cor}[thm]{Corollary}
\newtheorem{lem}[thm]{Lemma}
\newtheorem{pro}[thm]{Proposition}
\theoremstyle{definition}
\newtheorem{defn}[thm]{Definition}
\newtheorem{exmp}[thm]{Example}
\theoremstyle{myremark}
\newtheorem{rem}[thm]{Remark}
  
\numberwithin{equation}{section}

\begin{document}

\title[Approximations related to Tempered stable Distributions]{Approximations related to Tempered stable Distributions}

\author[Barman]{Kalyan Barman}
\address{\hskip-\parindent
Kalyan Barman, Department of Mathematics, IIT Bombay,
Powai - 400076, India.}

\email{barmankalyan@math.iitb.ac.in}

\author[Upadhye]{Neelesh S Upadhye}

\address{\hskip-\parindent
Neelesh S Upadhye, Department of Mathematics, IIT Madras,
Chennai - 600036, India.}
\email{neelesh@iitm.ac.in}

\author[Vellaisamy]{Palaniappan Vellaisamy}
\address{\hskip-\parindent
	Palaniappan Vellaisamy, Department of Statistics and Applied Probability, UC Santa Barbara, Santa Barbara, CA, 93106, USA}
\email{pvellais@ucsb.edu}


\begin{abstract}
In this article,  we first obtain, for the Kolmogorov distance, an error bound between a tempered stable and a compound Poisson distribution (CPD) and also an error bound between a tempered stable and an $\alpha$-stable distribution via Stein's method. For the smooth Wasserstein distance, an error bound between two tempered stable distributions (TSD) is also derived. As examples, we discuss the approximation of a TSD to normal and variance-gamma distributions (VGD). As corollaries, the corresponding limit theorem follows.
\end{abstract}

\maketitle

\section{Introduction}\label{Sec:Introduction}

\noindent	
Probability approximations is one of the fundamental topics in probability theory, due to its wide range of applications in limit theorems \cite{cekana,k1,vellai00}, runs \cite{k7}, stochastic algorithms \cite{wang}, and various other fields. They mainly provide estimates of the distance between the distributions of two random variables (rvs), which measure closeness of the approximations. Hence, estimating the accuracy of the approximation is a crucial task. Recently, Chen {\it et al.} \cite{k16,k17}, Jin {\it et al.} \cite{Jin}, Upadhye and Barman \cite{k1}, Xu \cite{k20} have studied stable approximations via the Stein's method. The distributional approximations for a family of stable distributions is not straightforward, due to the lack of symmetry and heavy-tailed behavior of stable distributions. One of the major obstacles is that the moments of a stable distribution do not exist, whenever the stability parameter $\alpha \in (0,1]$. To overcome these issues, different approaches and various assumptions are used.

\vskip 2ex
\noindent
Koponen \cite{koponen} first introduced tempered stable distributions (TSD)  by tempering the tail of the stable (also called $\alpha$-stable) distributions and making the distribution's tail lighter. The tails of the TSD are heavier than the normal distribution and thinner than the $\alpha$-stable distribution, see \cite{kim}. Therefore, quantifying the error in approximating  $\alpha$-stable and normal distributions to a TSD is of interest. A TSD has mean, variance and exponential moments. Also, the class of TSD includes many well-known sub-families of probability distributions, such as CGMY, KoBol, bilateral-gamma, and variance-gamma distributions, which have applications in several disciplines including financial mathematics, see \cite{k14,k15,k9,k10}.

\vskip 2ex
\noindent
In this article,  we first obtain, for the Kolmogorov distance, an error bound between tempered stable and compound Poisson distributions (CPD). This provides a convergence rate for the tempered stable approximation to a CPD. Next, we obtain the error bounds between tempered stable and $\alpha$-stable distributions via Stein's method. We obtain also the error bounds between two TSD's, for smooth Wasserstein distance. As a consequence, we discuss the normal and variance-gamma approximation and the corresponding limit theorems to a TSD.

\vskip 2ex
\noindent
The organization of this article is as follows. In Section \ref{pre}, we discuss some notations and preliminaries that will be useful later. First, we discuss some important properties of TSD and some special and limiting distributions from the TSD family. A brief discussion on Stein's method is also presented. In Section \ref{SMTSD}, we establish a Stein identity and a Stein equation for TSD and solve it via the semigroup approach. The properties of the solution to the Stein equation are discussed. In Section \ref{BFTSD}, we discuss bounds for tempered stable approximations for various probability distributions.  

\section{The preliminary results}\label{pre}
\subsection{Properties of tempered stable distributions}\label{TPOTST}
 We first define the TSD and discuss some of their the properties. Let $\textbf{I}_B(.)$ denotes the indicator function of the set $B$.
A rv $X$ is said to have TSD (see \cite[p.2]{k13}) if its cf is given by 
\begin{equation}\label{e1}
\phi_{ts}(z)=\exp\left(\int_{\mathbb{R}}(e^{izu}-1)\nu_{ts}(du)  \right),~~z\in\mathbb{R},
\end{equation}
where the L\'evy measure $\nu_{ts}$ is 
\begin{equation}\label{e2}
\nu_{ts}(du)=\left(\frac{m_1  }{u^{1+\alpha_1}}e^{-\lambda _1 u}\mathbf{I}_{(0,\infty)}(u)+\frac{m_2 }{|u|^{1+\alpha_2}}e^{-\lambda_2|u|}\mathbf{I}_{(-\infty,0)}(u)\right)du,
\end{equation}
with parameters $m_i,\lambda_i\in(0,\infty)$ and $\alpha_i\in[0,1)$, for $i=1,2$, and we denote it by $\text{TSD} (m_1,\alpha_1,\lambda_1,m_2,\alpha_2,\lambda_2)$. Note that TSD are infinitely divisible and self-decomposable, see \cite{k13}. Also, note that, if $\alpha_1=\alpha_2=\alpha \in (0,1),$ the L\'evy measure in \eqref{e2} can be seen as
\begin{align}
\nu_{ts}(du)= q(u)\nu_\alpha (du),
\end{align}
\noindent
where 
\begin{align}\label{levyst0}
\nu_\alpha(du)= \left(\frac{m_1}{u^{1+\alpha}}\mathbf{I}_{(0,\infty)}+ \frac{m_2}{u^{1+\alpha}}\mathbf{I}_{(-\infty,0)} \right)du
\end{align}
\noindent
is the L\'evy measure of a $\alpha$-stable distribution (see \cite{sato}) and $q:\mathbb{{R}} \to \mathbb{{R}}_{+}$ is a tempering function  (see \cite{k13}), given by
\begin{align}
q(u)=\bigg( e^{-\lambda _1 u}\mathbf{I}_{(0,\infty)}(u)+e^{-\lambda_2|u|}\mathbf{I}_{(-\infty,0)}(u)  \bigg).
\end{align}
\noindent
The following special and limiting cases of TSD are well known, see \cite{k13}. 

\noindent
Let $\overset{L}{\to}$ denote the convergence in distribution. Also let $m_i,\lambda_i\in(0,\infty)$ and $\alpha_i\in[0,1)$, for $i=1,2$.
\begin{itemize}
	\item [(i)] When $\alpha_1=\alpha_2=\alpha$, then $\text{TSD} (m_1,\alpha,\lambda_1,m_2,\alpha,\lambda_2)$ has KoBol distributions.
	\item [(ii)] When $m_1=m_2=m$ and  $\alpha_1=\alpha_2=\alpha$, then $\text{TSD} (m,\alpha,\lambda_1,m,\alpha,\lambda_2)$ has CGMY distributions.
	\item [(iii)] When $\alpha_1=\alpha_2=0$, then $\text{TSD} (m_1,0,\lambda_1,m_2,0,\lambda_2)$ has bilateral-gamma distributions (BGD), denoted by BGD$(m_1,\lambda_1,m_2,\lambda_2)$.
	
	
	\item [(iv)] When $m_1=m_2=m$ and $\alpha_1=\alpha_2=0$, then $\text{TSD} (m,0,\lambda_1,m,0,\lambda_2)$ has variance-gamma distributions (VGD), denoted by VGD$(m,\lambda_1,\lambda_2)$.
	
	\item [(v)] When $m_1=m_2=m$, $\lambda_1=\lambda_2=\lambda$ and $\alpha_1=\alpha_2=0$, then $\text{TSD} (m,0,\lambda,m,0,\lambda)$ has symmetric VGD, denoted by SVGD$(m,\lambda)$.
	
	\item [(vi)] When $\lambda_1,\lambda_2 \downarrow 0,$ then $\text{TSD} (m_1,\alpha,\lambda_1,m_2,\alpha,\lambda_2)$ converges to an $\alpha$-stable distribution, denoted by $S (m_1,m_2,\alpha)$, with cf
	\begin{equation}\label{0e1}
	\phi_{\alpha}(z)=\exp\left(\int_{\mathbb{R}}(e^{izu}-1)\nu_{\alpha}(du)  \right),~~z\in\mathbb{R},
	\end{equation}
	\noindent
	where $\nu_{\alpha}$ is the L\'evy measure given in \eqref{levyst0}.

	\item [(vii)] The limiting case as $m \to \infty $, the SVGD$(m,\sqrt{2m}/\lambda)$ has the normal $\mathcal{N}(0,\lambda^2)$ distribution.
\end{itemize}

\noindent
Let $X\sim \text{TSD}(m_1,\alpha_1,\lambda_1,m_2,\alpha_2,\lambda_2)$. Then from \eqref{e1}, for $z\in \mathbb{{R}}$, the cumulant generating function is given by
\begin{align}
\Psi (z)=\log \mathbb{E}(e^{izX})= \log \phi_{ts} (z),
\end{align}
\noindent
where $\phi(z)$ is given in \eqref{e1}. Then the $n$-th cumulant of $X$ is
\begin{align}\label{cumulfor}
C_n(X)&:=(-i)^n \bigg[\frac{d^n}{dz^n}\Psi (z)\bigg]_{z=0}=\displaystyle\int_{\mathbb{R}} u^n \nu_{ts} (du)<\infty,~n\geq 1.
\end{align}
\noindent
In particular (see \cite{k13}), 
\begin{align}
C_1(X)&=\mathbb{E}(X)=\Gamma (1- \alpha_1) \frac{m_1}{\lambda_1^{1-\alpha_1}}-\Gamma (1- \alpha_2) \frac{m_2}{\lambda_2^{1-\alpha_2}},\\
C_2(X)&=Var(X)=\Gamma (2- \alpha_1) \frac{m_1}{\lambda_1^{2-\alpha_1}}+\Gamma (2- \alpha_2) \frac{m_2}{\lambda_2^{2-\alpha_2}},\\
\text{and }C_3(X)&=\Gamma (3- \alpha_1) \frac{m_1}{\lambda_1^{3-\alpha_1}}-\Gamma (3- \alpha_2) \frac{m_2}{\lambda_2^{3-\alpha_2}}.
\end{align}
\subsection{Key steps of Stein's method}
 Let $f^{(n)}$ henceforth denotes the $n$-th derivative of $f$ with $f^{(0)}=f$ and $f^{\prime}=f^{(1)}$. Let $\mathcal{S}(\mathbb{R})$ be the Schwartz space defined by
\begin{align}\label{schspace}
\mathcal{S}(\mathbb{R}):=\left\{f\in C^\infty(\mathbb{R}): \lim_{|x|\rightarrow \infty} |x^mf^{(n)}(x)|=0, \text{ for all } m,n\in \mathbb{N}_0\right\},
\end{align}
\noindent
where $\mathbb{N}_0=\mathbb{N}\cup \{ 0\}$ and $C^\infty(\mathbb{R})$ is the class of infinitely differentiable functions on $\mathbb{R}$. Note that the Fourier transform (FT) on $\mathcal{S}(\mathbb{R})$ is automorphism. In particular, if $f \in \mathcal{S}(\mathbb{R})$, and $\widehat{f}(u):=\int_{\mathbb{R}}e^{-iux}f(x)dx,~~\text{  }u\in\mathbb{R},$ then $\widehat{f}(u)\in \mathcal{S}(\mathbb{R}).$ Similarly, if $\widehat{f}(u)\in \mathcal{S}(\mathbb{R})$, and $f(x):=\frac{1}{2\pi}\int_{\mathbb{R}}e^{iux}\widehat{f}(u)du,~~\text{  }x\in\mathbb{R},$
then $f(x)\in \mathcal{S}(\mathbb{R})$, see \cite{stein}.

\vskip 2ex
\noindent
Next, let
\begin{align}\mathcal{H}_r = \{h:\mathbb{R}\to \mathbb{R} | h \mbox{ is $r$ times differentiable and},\|h^{(k)}\|\leq 1, k =0, 1,\ldots, r \},\end{align} where $\|h\|=\sup_{x\in\mathbb{R}}|h(x)|$. Then, for any two rvs $Y$ and $Z$, the smooth Wasserstein distance (see \cite{rg0}) is given by
\begin{equation}\label{SWD00}
d_{\mathcal{H}_r}(Y,Z):=\sup_{h \in \mathcal{H}_r}\left|\mathbb{E}[h(Y)]-\mathbb{E}[h(Z)]\right|,~r\geq 1.
\end{equation}
\noindent
 Also, let  
 \begin{align}
 \mathcal{H}_W = \{h:\mathbb{R}\to \mathbb{R} | h \mbox{ is $1$-Lipschitz and},\|h\|\leq 1 \}.
 \end{align}
 \noindent
 Then, for any two rvs $Y$ and $Z$, the classical Wasserstein distance (see \cite{rg0}) is given by
 \begin{equation}
 d_{W}(Y,Z):=\sup_{h \in \mathcal{H}_W}\left|\mathbb{E}[h(Y)]-\mathbb{E}[h(Z)]\right|.
 \end{equation}
\noindent
Finally, let
\begin{align}
	\mathcal{H}_K &= \left\{h:\mathbb{R}\to \mathbb{R}~\big|~h=\mathbf{1}_{(-\infty,x]},~x\in \mathbb{{R}} \right\}.
\end{align}
\noindent
Then, for any two rvs $Y$ and $Z$, the Kolmogorov distance (see \cite{rg0}) is given by
\begin{equation}
d_{K}(Y,Z):=\sup_{h \in \mathcal{H}_K}\left|\mathbb{E}[h(Y)]-\mathbb{E}[h(Z)]\right|.
\end{equation}


\vskip 2ex
\noindent
Next, we discuss the steps of Stein's method. The method is based on the simple fact that, any real-valued random variable (rv) $Z$ has probability distribution $F_Z$ (denoted by $Z\sim F_Z$), if and only if 
\begin{equation*}
\mathbb{E}\left(\mathcal{A}f(Z) \right)=0,
\end{equation*}

\noindent
where $f \in \mathcal{F}$, a class suitable functions. This equivalence is called Stein characterization of $F_Z$. This characterization leads us to the Stein equation
\begin{equation}\label{normal2}
Af(x)=h(x)-\mathbb{E}[h(Z)],
\end{equation}
\noindent
where $h$ is a real-valued test function. Replacing $x$ with a rv $Y$ and taking expectations on both sides of  (\ref{normal2}) gives
\begin{equation}\label{normal3}
\mathbb{E}\left[\mathcal{A} f (Y)\right]=\mathbb{E}h(Y)-\mathbb{E}h(Z).
\end{equation}

\noindent 
This equality (\ref{normal3}) plays a crucial role in Stein's method. The probability distribution $F_Z$ is characterized by (\ref{normal2}) such that the problem of bounding the quantity $|\mathbb{E}h(Y)-\mathbb{E}h(Z)|$ depends on smoothness of the solution to (\ref{normal2}), and behavior of $Y$. For more details on Stein's method, we refer to the reader \cite{k0,appmethod} and the references therein.
 
 \vskip 2ex
\noindent 
In particular, let $Z$ has normal $\mathcal{N}(0,\sigma^2)$ distribution. Then a Stein characterization for $Z$ (see \cite{k2}) is  
 \begin{equation}\label{normal1}
 \mathbb{E}\left(\sigma^{2}f^{\prime}(Z)-Zf(Z) \right)=0,
 \end{equation}
 
 \noindent
 where $f$ is any real-valued absolutely continuous function such that $\mathbb{E}|f^{\prime}(Z)|<\infty$. This characterization leads us to the Stein equation
 \begin{equation}\label{normal02}
 \sigma^{2}f^{\prime}(x)-xf(x)=h(x)-\mathbb{E}h(Z),
 \end{equation}
 \noindent
 where $h$ is a real-valued test function. Replacing $x$ with a rv $Z_n \sim \mathcal{N}(0, \sigma^{2}_{n})$ and taking expectations on both sides of  (\ref{normal02}) gives
 \begin{equation}\label{normal03}
 \mathbb{E}\left(\sigma^{2}f^{\prime}(Z_n)-Z_nf(Z_n) \right)=\mathbb{E}h(Z_n)-\mathbb{E}h(Z).
 \end{equation}
 
 \noindent
 Using the smoothness of solution to (\ref{normal02}), it can be shown (see \cite[Section 3.6]{nourdin}) that
 \begin{align}\label{dbtwn}
 	d_W(Z_n,Z) \leq \frac{\sqrt{2/\pi}}{\sigma^2 \vee \sigma^{2}_{n}} |\sigma^{2}_{n}-\sigma^{2}|.
 \end{align}
 \noindent
  From \eqref{dbtwn}, if $\sigma_n \to \sigma$, then $d_W(Z_n,Z)=0,$ as expected, which implies that $Z_n$ converges to a normal $\mathcal{N}(0,\sigma^2)$ distribution. We refer the reader to \cite{amit0} and \cite{ley} for a number of similar bounds as \eqref{dbtwn} for comparison of univariate probability distributions.



\section{Stein's method for tempered stable distributions}\label{SMTSD}

\subsection{A Stein identity for tempered stable distributions}
In this section, we obtain a Stein identity for a TSD. First recall that $\mathcal{S}(\mathbb{{R}})$ denotes the Schwartz space of functions, defined in \eqref{schspace}.
\begin{pro}\label{th1}
	A rv $X\sim \text{TSD} (m_1,\alpha_1,\lambda_1,m_2,\alpha_2,\lambda_2)$, if and only if
	\begin{equation}\label{PP2:StenIdTSD}
	\mathbb{E}\left(Xf(X)-\displaystyle\int_{\mathbb{R}}uf(X+u)\nu_{ts}(du) \right)=0,~~f\in\mathcal{S}(\mathbb{R}),  
	\end{equation}
	 where $\nu_{ts}$ is the associated L\'evy measure of TSD, defined in \eqref{e2}.	
\end{pro}
\begin{proof}
Taking logarithms on both sides of (\ref{e1}), and differentiating with respect to $z$, we have
\begin{equation}\label{PP2:e4}
\phi_{ts}^{\prime}(z)=i\phi_{ts}(z)\int_{\mathbb{R}}ue^{izu}\nu_{ts}(du).
\end{equation}

\noindent
Let $F_{X}$ be the cumulative distribution function (CDF) of $X$. Then,

\begin{equation}\label{PP2:e5}
\phi_{ts}(z)=\displaystyle\int_{\mathbb{R}}e^{izx}F_{X}(dx)~~\implies~~\phi_{ts}^{\prime}(z)=i\displaystyle\int_{\mathbb{R}}xe^{izx}F_{X}(dx).
\end{equation}
\noindent
Using \eqref{PP2:e5} in \eqref{PP2:e4} and rearranging the integrals, we have
\begin{align}
\nonumber	0&=i\displaystyle\int_{\mathbb{R}}xe^{izx}F_{X}(dx)-i\phi_{ts}(z)\int_{\mathbb{R}}ue^{izu}\nu_{ts}(du)\\
	&=\displaystyle\int_{\mathbb{R}}xe^{izx}F_{X}(dx)-\phi_{ts}(z)\int_{\mathbb{R}}ue^{izu}\nu_{ts}(du)\label{PP2:e6}
\end{align}

\noindent
The second integral of \eqref{PP2:e6} can be written as

\begin{align}
	\nonumber \phi_{ts}(z)\int_{\mathbb R}ue^{izu}\nu_{ts}(du)&=\int_{\mathbb R}\int_{\mathbb R}ue^{izu}e^{izx}F_{X}(dx)\nu_{ts}(du)\\
	\nonumber &=\int_{\mathbb R}\int_{\mathbb R}ue^{iz(u+x)}\nu_{ts}(du)F_{X}(dx)\\
	\nonumber &=\int_{\mathbb R}\int_{\mathbb R}ue^{izy}\nu_{ts}(du)F_{X}(d(y-u))\\
	\nonumber &=\int_{\mathbb R}\int_{\mathbb R}ue^{izx}\nu_{ts}(du)F_{X}(d(x-u))\\ 
	&=\int_{\mathbb R}e^{izx}\int_{\mathbb R}uF_{X}(d(x-u))\nu_{ts}(du).\label{PP2:e7}
\end{align}

\noindent
Substituting \eqref{PP2:e7} in \eqref{PP2:e6}, we have

\begin{align}
\nonumber	0&=\displaystyle\int_{\mathbb{R}}xe^{izx}F_{X}(dx)-\int_{\mathbb R}e^{izx}\int_{\mathbb R}uF_{X}(d(x-u))\nu_{ts}(du)\\
	&=\displaystyle\int_{\mathbb{R}}e^{izx} \left( xF_{X}(dx)-\int_{\mathbb R}uF_{X}(d(x-u))\nu_{ts}(du) \right) \label{PP2:e8}
\end{align}

\noindent
On applying Fourier transform to \eqref{PP2:e8}, multiplying with $f\in \mathcal{S}(\mathbb{R}),$ and integrating over $\mathbb{R},$ we get
\begin{align}\label{PP2:e9}
	\displaystyle\int_{\mathbb{R}}f(x) \left( xF_{X}(dx)-\int_{\mathbb R}uF_{X}(d(x-u))\nu_{ts}(du) \right)=0.
\end{align}

\noindent
The second integral of \eqref{PP2:e9} can be seen as

\begin{align}
	\nonumber \int_{\mathbb R}\int_{\mathbb R}uf(x)F_{X}(d(x-u))\nu_{ts}(du)&=\int_{\mathbb R}\int_{\mathbb R}uf(y+u)F_{X}(dy)\nu_{ts}(du)\\ 
	\nonumber &=\int_{\mathbb R}\int_{\mathbb R}uf(x+u)F_{X}(dx)\nu_{ts}(du)\\
	&=\mathbb{E}\left(\int_{\mathbb R}uf(X+u)\nu_{ts}(du)\right).\label{PP2:e10}
\end{align}
\noindent
Substituting \eqref{PP2:e10} in \eqref{PP2:e9}, we have
\begin{align*}
	\mathbb{E}\left(Xf(X)-\displaystyle\int_{\mathbb{R}}uf(X+u)\nu_{ts}(du) \right)=0,
\end{align*}
\noindent
which proves \eqref{PP2:StenIdTSD}.
\vfill
\noindent
Assume conversely, \eqref{PP2:StenIdTSD} holds for $\nu_{ts}$ defined in \eqref{e2}. For any $s \in \mathbb{R}$, let $f(x)=e^{isx}$, $x\in \mathbb{R}$, then (\ref{PP2:StenIdTSD}) becomes
\begin{align*}
\mathbb{E}Xe^{isX} &= \mathbb{E}\int_{\mathbb R}e^{is(X+u)}u\nu(du)\\
&=\mathbb{E}e^{isX}\int_{\mathbb R}e^{isu}u\nu(du).
\end{align*}
\noindent
Setting $\phi_{ts}(s)=\mathbb{E}e^{isX}$, then
\begin{equation}\label{suff1}
\phi_{ts}^{\prime}(s)=i\phi_{ts}(s)\int_{\mathbb R}e^{isu}u\nu(du).
\end{equation}

\noindent
Integrating out the real and imaginary parts of (\ref{suff1}) leads, for any $z\geq 0$, to
\begin{align*}
\phi_{ts}(z)&=\exp \left(i\int_{0}^{z}\int_{\mathbb R}e^{isu}u\nu(du)ds  \right)\\
&=\exp\left(i\int_{\mathbb R}\int_{0}^{z}e^{isu}dsu\nu(du)  \right)\\
&=\exp\left(\int_{\mathbb R}(e^{izu}-1)\nu(du)  \right).
\end{align*}

\noindent
A similar computation for $z\leq0$ completes the derivation of the cf. 
\end{proof}
\noindent
We now have the following Corollary for $\alpha$-stable distributions.
\begin{cor}
	A rv $X \sim S(m_1,m_2,\alpha)$, if and only if 
	\begin{equation}\label{StenIdstable}
	\mathbb{E}\left(Xf(X)-\displaystyle\int_{\mathbb{R}}uf(X+u)\nu_\alpha(du) \right)=0,~~f\in\mathcal{S}(\mathbb{R}),  
	\end{equation}
	where $\nu_{\alpha}$ is the associated L\'evy measure of an $\alpha$-stable distribution, given in \eqref{levyst0}.
\end{cor}
\begin{proof}
	Let $\alpha_1=\alpha_2=\alpha$ in Theorem \ref{th1}. Next, taking limits as $\lambda_1,\lambda_2 \downarrow 0$ in \eqref{PP2:StenIdTSD}, and then applying the dominated convergence theorem, and noting that $\nu_{ts} \to \nu_{\alpha},$ we get \eqref{StenIdstable}.
\end{proof}
\begin{rem}	 
(i) Note that we derive the characterizing (Stein) identity (\ref{PP2:StenIdTSD}) for TSD using the L\'evy-Khinchine representation of the cf. Also, observe that several classes of distributions such as variance-gamma, bilateral-gamma, CGMY, and KoBol can be viewed as TSD. Stein identities for these classes of distributions can be easily obtained using (\ref{PP2:StenIdTSD}).

\item [(ii)] Recently Arras and Houdr\'e [\cite{k0}, Theorem 3.1 and Section 5] obtained a Stein identity for infinitely divisible distributions with first finite moment via the covariance representation given in \cite{interpol}. Note that TSD is a subclass of IDD and TSD has finite mean. Hence, Stein identity for TSD can also be derived using the approach given in \cite{k0}.	
\end{rem}
\subsubsection{A non-zero bias distribution}
In the Stein's method literature, the zero bias distribution is a powerful tool to obtain bounds, which has been used in several situation. It has been used in conjunction with coupling techniques to produce quantitative results for normal and product normal approximations, see, e.g., \cite{kk3}. The zero bias distribution due to Goldstein and Reinert \cite{GR0} is as follows.
\begin{defn}
	Let X be a rv with $\mathbb{E}(X)=0$, and $Var(X)=\sigma^2<\infty$. We say that $X^*$ has $X$-zero bias distribution if
	\begin{align}
		\mathbb{E}[Xf(X)]=\sigma^2 \mathbb{E}[f^\prime (X^*)],
	\end{align}
	for any differentiable function $f$ with $\mathbb{E}(Xf(X)) < \infty$.
\end{defn}
\noindent
In the following lemma, we prove existence of a non-zero (extended) bias distribution (see \cite[Remark 3.9 (ii)]{k0}) associated with TSD. Before stating our result, let us define
\begin{align}\label{qnnot1}
\eta^{+}(u)=\int_{u}^{\infty}y \nu_{ts}(dy), ~u>0,\text{ and }\eta^{-}(u)=\int_{-\infty}^{u}(-y) \nu_{ts}(dy),~u<0,
\end{align}
\noindent
where $\nu_{ts}$ is the L\'evy measure of the TSD (see \eqref{e2}). Let
$$\eta(u):=\eta^{+}(u)+\eta^{-}(u). $$
\noindent
 Also let $Y$ be a random variable with the density
\begin{align}\label{pdfzbias}
f_{1}(u)=\frac{\eta(u)}{\int_{\mathbb{R}} y^{2} \nu_{ts} (dy)}=\frac{\eta(u)}{Var(X)}, ~u\in\mathbb{R}.
\end{align}
\noindent
Then, for $n\geq 1$, the $n$th moment of $Y$ is 
\begin{align}
\nonumber	\mathbb{E}(Y^n)&=\frac{1}{Var(X)}\displaystyle\int_{\mathbb{R}}u^n\eta(u)du\\
\nonumber	&=\frac{1}{(n+1)Var(X)}\displaystyle\int_{\mathbb{R}}u^{n+2}\nu_{ts}(u)du\\
	\label{mom0}&=\frac{C_{n+2}(X)}{(n+1)C_2 (X)}.
\end{align} 
\begin{lem}\label{theorem1}
	Let $X\sim \text{TSD}(m_1,\alpha_1,\lambda_1,m_2,\alpha_2,\lambda_2)$and $Y$ (independent of $X$) has the density given in \eqref{pdfzbias}. Then
	\begin{equation}\label{eqn1}
	\text{Cov}(X,f(X))=Var(X)\mathbb{E} \left(f^{\prime}(X+Y)\right),
	\end{equation}
	where $g$ is an absolutely continuous function with $\mathbb{E} \left(f^{\prime}(X+Y)\right)<\infty.$
\end{lem}
\begin{proof}
 Using \eqref{cumulfor}, for the case $n=1$, in \eqref{PP2:StenIdTSD}, and rearranging the terms, we get 
$$	\text{Cov}(X,f(X))=\mathbb{E}\int_{\mathbb{R}} u(f(X+u)-f(X))\nu_{ts}(du).$$
\noindent
 Now
\begin{align}
\nonumber\text{Cov}(X,f(X))=&\mathbb{E}\int_{\mathbb{R}} u(f(X+u)-f(X))\nu_{ts}(du)\\
\nonumber=&\mathbb{E}\displaystyle\int_{0}^{\infty}f^{\prime}(X+v)\int_{v}^{\infty}u\nu_{ts}(du) dv\\
\nonumber&+\mathbb{E}\displaystyle\int_{-\infty}^{0}f^{\prime}(X+v)\int_{-\infty}^{v}(-u)\nu_{ts}(du) dv\\
\nonumber	=&\mathbb{E}\displaystyle\int_{\mathbb{R}}f^{\prime}(X+v) \left(\eta^{+}(v)\mathbf{I}_{(0,\infty)}(v)+ \eta^{-}(v)\mathbf{I}_{(-\infty,0)}(v)  \right)dv\\
\nonumber=&\left(\displaystyle\int_{\mathbb{R}} u^{2}\nu_{ts} (du) \right)\mathbb{E}\displaystyle\int_{\mathbb{R}}f^{\prime}(X+v)f_{1}(v)dv\\
\nonumber=&\left(\displaystyle\int_{\mathbb{R}} u^{2}\nu_{ts} (du) \right) \mathbb{E} \left(f^{\prime}(X+Y)\right)\\
 \nonumber=&Var(X)\mathbb{E} \left(f^{\prime}(X+Y)\right).
\end{align}
\noindent
This proves the result.
\end{proof}

\begin{rem}
	Note that the covariance identity in \eqref{eqn1} coincides with the one given in \cite[Proposition 3.8]{k0}. However, the usefulness of the identity is shown in deriving the error bound of the limiting distributions of TSD (see the proof of Theorem \ref{bd:th2}).
\end{rem}

\subsection{A Stein equation for tempered stable distributions} In this section, we first derive a Stein equation for TSD and then solve it via the semigroup approach. From Proposition \ref{th1}, for any $f \in \mathcal{S}(\mathbb{R})$,
\begin{align}\label{STOPTSD}
	\mathcal{A}f(x):=-xf(x)+\displaystyle\int_{\mathbb{R}}uf(x+u)\nu_{ts}(du)
\end{align} 
\noindent
is a Stein operator for TSD. Hence, a Stein equation for $X\sim$TSD$(m_1,\alpha_1,\lambda_1,m_2,\alpha_2,$ $\lambda_2)$ is given by
\begin{equation}\label{PP2:e15}
	\mathcal{A}f(x)= h(x)-\mathbb{E}(h(X)),
\end{equation}
where $h\in \mathcal{H},$ a class of test functions. The semigroup approach for solving the Stein equation \eqref{PP2:e15} is developed by Barbour \cite{k8}, and Arras and Houdr\'e \cite{k0} generalized it for infinitely divisible distributions with the finite first moment. Let $\mathcal{F}=\bar{\mathcal{S}}(\mathbb{{R}})$ henceforth denote the closure of $\mathcal{S}(\mathbb{{R}})$. Consider the following family of operators $(P_{t})_{t\geq0}$, for all $x\in\mathbb{R}$, as
\begin{equation}\label{PP2:e16}
	P_{t}(f)(x)=\frac{1}{2\pi}\int_{\mathbb{R}}\hat{f}(z)e^{iz xe^{-t}}\frac{\phi_{ts}(z)}{\phi_{ts}(e^{-t}z)}dz, ~~f\in\mathcal{F},
\end{equation} 
\noindent
  where $\hat{f}$ is the FT of $f$, and $\phi_{ts}$ is the cf of TSD given in \eqref{e1}. It is well known that (see \cite{sato}), one can define a cf, for all $z\in \mathbb{R},$ and $t\geq 0,$ by
\begin{align}\label{PP2:a15}
	\phi_t(z):=\frac{\phi_{ts}(z)}{\phi_{ts}(e^{-t}z)}=\displaystyle\int_{\mathbb{R}}e^{iz u}F_{X_{(t)}}(du),
\end{align}
\noindent
where $F_{X_{(t)}}$ is the CDF of $X_{(t)}$. The property given in \eqref{PP2:a15} is also known as self-decomposability, see \cite{sato}. Using this property, we get

\begin{align}
	\nonumber  P_t(f)(x)&=\frac{1}{2\pi}\int_{\mathbb R}\int_{\mathbb{R}}\widehat{f}(z)e^{iz xe^{-t}}e^{iz u}F_{X_{(t)}}(du)dz\\
	\nonumber  &=\frac{1}{2\pi}\int_{\mathbb R}\int_{\mathbb{R}}\widehat{f}(z)e^{iz (u+xe^{-t})}F_{X_{(t)}}(du)dz\\
	&=\displaystyle\int_{\mathbb{R}}f(u+xe^{-t})F_{X_{(t)}}(du),\label{PP2:a17}
\end{align}
\noindent
where the last step follows by applying inverse FT.

\begin{pro}\label{PP2:proSem}
	Let $(P_{t})_{t\geq 0}$ a family of operators given in \eqref{PP2:e16}. Then
	\begin{enumerate}
	\item [(i)]$(P_{t})_{t\geq 0}$ is a $\mathbb{C}_0$-semigroup on $\mathcal{F}$.
	\item [(ii)]Its generator $T$ is given by
	\begin{equation}\label{e7}
	T(f)(x)=-xf^{\prime}(x)+\displaystyle\int_{\mathbb R}uf^{\prime}(x+u)\nu_{ts}(du),~~f\in\mathcal{S}(\mathbb{R}).
	\end{equation}	
	\end{enumerate}	 
\end{pro}
\noindent
Following the steps similar to Proposition 3.8 and Lemma 3.10 of \cite{k1}, the proof follows.

\vskip 1ex
\noindent
Next, we provide a solution to Stein equation \eqref{PP2:e15}.  

\begin{thm}\label{thmsol}
	Let $X\sim \text{TSD}(m_1,\alpha_1,\lambda_1,m_2,\alpha_2,\lambda_2)$and $h\in\mathcal{H}_{r}$. Then the function $f_{h}:\mathbb{R}\to \mathbb{R}$ defined by 
\begin{equation}\label{PP2:SolSe}
f_{h}(x):=-\displaystyle
\int_{0}^{\infty}\frac{d}{dx}P_{t}h(x)dt,
\end{equation}
solves \eqref{PP2:e15}.

\end{thm}

\begin{proof}
Let $$g_{h}(x)=-\displaystyle\int_{0}^{\infty}\bigg(P_{t}(h)(x)-\mathbb{E}h(X)  \bigg)dt .$$
\noindent
Then $g_{h}^{\prime}(x)=f_{h}(x).$ Also from \eqref{e7}, we get
\begin{align}
\nonumber\mathcal{A}f_{h}(x)&=-xf_{h}(x)+\int_{\mathbb{R}} uf_{h}(x+u)\nu_{bg} (du)\\
\nonumber&=Tg_{h}(x) \\
\nonumber&=-\displaystyle\int_{0}^{\infty}TP_{t}(h)(x)dt\\
\nonumber&=-\displaystyle\int_{0}^{\infty}\frac{d}{dt}P_{t}h(x)dt\text{ (see \cite[p.68]{nourdin})}\\
\nonumber&=P_{0}h(x)-P_{\infty}h(x)\\
\nonumber&=h(x)-\mathbb{E}h(X)~(\text{by Proposition \ref{PP2:proSem}}).
\end{align}
\noindent
Hence, $f_{h}$ is the solution to \eqref{PP2:e15}. 	
\end{proof}
\subsection{Properties of the solution of the Stein equation}
The next step is to estimate the properties of $f_{h}$. In the following theorem, we establish estimates of $f_{h}$, which play a crucial role in the TSD approximation problems. Gaunt \cite{k24,kk2} and D$\ddot{o}$bler {\it et al.} \cite{kk6} propose various methods for bounding the solution to the Stein equations that allow them to derive properties of the solution to the Stein equation, in particular for a sub-family of TSD, namely the variance-gamma.
\begin{lem}\label{th3}
		For $h\in\mathcal{H}_{r+1}$, let $f_{h}$ be defined in \eqref{PP2:SolSe}. Then
		\begin{enumerate}
			\item [(i)] for $r=0,1,2,\ldots,$			
			\begin{align}\label{PP2:pr1}
			\|f_{h}^{(r)}\|\leq \frac{1}{r+1} \|h^{(r+1)}\|.
			\end{align}
			\item [(ii)] For any $x,y\in\mathbb{R},$ 
			\begin{align}\label{PP2:pr02}
			\|f^{\prime}_{h}(x)- f^{\prime}_{h}(y)  \| \leq \frac{\|h^{(3)}\|}{3}\left| x-y\right|.
			\end{align}	 
		\end{enumerate}

		\end{lem}

\begin{proof}
(i) For $h\in \mathcal{H}_{r+1}$,
\begin{align*}
\|f_h\|&=\sup_{x\in\mathbb{R}} \left|-\displaystyle
\int_{0}^{\infty}\frac{d}{dx}P_{t}h(x)dt\right|\\
&=\sup_{x\in\mathbb{R}} \left| -\displaystyle
\int_{0}^{\infty}e^{-t}\int_{\mathbb{R}}h^{(1)}(xe^{-t}+y)F_{X_{(t)}}(dy)dt \right|\\
&\leq \|h^{(1)}\| \left|\int_{0}^{\infty}e^{-t}dt \right|\\	
&= \|h^{(1)}\|.
\end{align*}
\noindent
It can be easily seen that $f_{h}$ is $r$-times differentiable. Let $r=1$, then 
\begin{align*}
\|f_h^{(1)}\|&=\sup_{x\in\mathbb{R}} \left| -\displaystyle
\int_{0}^{\infty}e^{-2t}\int_{\mathbb{R}}h^{(2)}(xe^{-t}+y)F_{X_{(t)}}(dy)dt \right|\\
&\leq \|h^{(2)}\| \left|\int_{0}^{\infty}e^{-2t}dt \right|\\	
&= \frac{\|h^{(2)}\|}{2}.
\end{align*}
\noindent
Also, by induction, we get
$$\|f_{h}^{(r)}\|\leq \frac{1}{r+1}\|h^{(r+1)}\|,~r=0,1,2,\ldots.$$

\vskip 1ex
\noindent
(ii) For any $x,y\in\mathbb{R}$ and $h\in \mathcal{H}_{3}$,
\begin{align*}
	\left|f^{\prime}_{h}(x)-f^{\prime}_{h}(y) \right| & \leq \displaystyle\int_{0}^{\infty}e^{-2t}\int_{\mathbb{R}}\left|h^{(2)}(xe^{-t}+z)-h^{(3)}(ye^{-t}+z)  \right|F_{X_{(t)}}(dz)dt\\
	&\leq  \displaystyle\int_{0}^{\infty}e^{-2t}\int_{\mathbb{R}}\|h^{(3)}\|\left|x-y\right|e^{-t}F_{X_{(t)}}(dz)dt\\
	&=\|h^{(3)}\|\left|x-y  \right| \displaystyle\int_{0}^{\infty}e^{-3t}dt\\
	&=\frac{\|h^{(3)}\|}{3}\left|x-y  \right|.
\end{align*}
\noindent
This proves the result.
\end{proof}

\vfill
\section{Bounds for tempered stable approximation}\label{BFTSD}
\noindent
In this section, we present bounds for the tempered stable approximations to various probability distributions. First, we obtain, for the Kolmogorov distance $d_K$, the error bounds for a sequence of CPD that converges to a TSD.  
\begin{thm}
	Let $X \sim \text{TSD} (m_1,\alpha_1,\lambda_1,m_2,\alpha_2,\lambda_2)$ and $X_n,$ $n\geq 1$ be a compound Poisson rvs with cf
		\begin{align}\label{cpdcf00}
		\phi_n(z):= \exp\bigg(n \left( \phi_{ts}^{\frac{1}{n}}(z)-1  \right)   \bigg),~z\in\mathbb{{R}},
		\end{align}
where $\phi_{ts}(t)$ is given in \eqref{e1}. Then
\begin{align}
	d_{K}(X_n,X) \leq c \bigg(\sum_{j=1}^{2}|C_{j}(X)|\bigg)^{\frac{2}{5}} \bigg(\frac{1}{n}\bigg)^{\frac{1}{5}},
\end{align}
where $c>0$ is independent of $n$ and $C_j$ denotes the $j$th cumulant of $X$.	
\end{thm} 
\begin{proof}
Let $b=\int_{-1}^{1}u\nu_{ts}(du)$, where $\nu_{ts}$ is defined in \eqref{e2}. Then by Equation (2.6) of \cite{k0},
$$\text{TSD} (m_1,\alpha_1,\lambda_1,m_2,\alpha_2,\lambda_2)\overset{d}{=}ID(b,0,\nu_{ts}).$$
\noindent
That is, $\text{TSD} (m_1,\alpha_1,\lambda_1,m_2,\alpha_2,\lambda_2)$ is an infinitely divisible distribution with the triplet $b,0$ and $\nu_{ts}$. Note that TSD are absolutely continuous with respect to the L\'evy measure with a bounded density and $\mathbb{E}|X|^2 < \infty$ (see \cite[Section 7]{k13}). Recall from Proposition 4.11 of \cite{k0} that, if $X\sim ID(b,0,\nu)$ with cf $\phi_X$ (say), and $X_n,$ $n\geq 1$ are compound Poisson rvs each of with cf as \eqref{cpdcf00}, then
\begin{align}\label{kolbdd}
	d_{K}(X_n,X)&\leq c\bigg( |\mathbb{E}(X)|+\int_{\mathbb{R}}u^2\nu(du)  \bigg)^{\frac{2}{p+4}}\bigg(\frac{1}{n}\bigg)^{\frac{1}{p+4}},
\end{align}
\noindent
where $|\phi_X(z)| \displaystyle\int_{0}^{|z|}\frac{ds}{|\phi_X(s)|} \leq c_0 |z|^p,$ $p\geq 1$.

\vskip 1ex
\noindent
Now observe that
\begin{align*}
	|\phi_{ts}(z)| \displaystyle\int_{0}^{|z|}\frac{ds}{|\phi_{ts}(s)|}
	&\leq \displaystyle\int_{0}^{|z|}\frac{ds}{|\mathbb{E}(\cos sX)+ i \mathbb{E}(\sin sX)|}\\	
	&=\displaystyle\int_{0}^{|z|}\frac{|\mathbb{E}(e^{\left( -is X\right)}) |}{\mathbb{E}^2 (\cos sX)+ \mathbb{E}^2 (\sin sX)}ds\\
	 &\leq c_0 \displaystyle\int_{0}^{|z|}|\mathbb{E}(e^{\left( -is X\right)}) |ds\\
	 &\quad\quad\quad\quad\bigg(\frac{1}{\mathbb{E}^2 (\cos sX)+ \mathbb{E}^2 (\sin sX)}<c_0 \text{, say}\bigg)\\
	& =c_0\displaystyle\int_{0}^{|z|}|\mathbb{E}(\cos sX-i\sin sX)|ds\\
	&\leq c_0\displaystyle\int_{0}^{|z|}ds\\
	&=c_0|z|.
\end{align*}
\noindent
Hence by \eqref{kolbdd}, for $p=1$, we get
\begin{align*}
d_{K}(X_n,X)&\leq c\bigg( |\mathbb{E}(X)|+\int_{\mathbb{R}}u^2\nu_{ts}(du)  \bigg)^{\frac{2}{5}}\bigg(\frac{1}{n}\bigg)^{\frac{1}{5}}\\
 &= c \bigg(\sum_{j=1}^{2}|C_{j}(X)|\bigg)^{\frac{2}{5}} \bigg(\frac{1}{n}\bigg)^{\frac{1}{5}}, 
\end{align*}

\noindent
$\text{ since }\int_{\mathbb{R}}u^2\nu_{ts}(du)=C_2(X).$ This proves the result.
\end{proof}
\begin{rem}
	Note that if $n\to \infty$, $d_K(X_n,X)=0$, as expected, and TSD is the limit of CPD.
\end{rem}
\noindent
Our next result yields an error bound for tempered stable approximation to $\alpha$-stable distributions.
 \begin{thm}
 	Let $\alpha \in (0,1)$. Let $X\sim$TSD$(m_1,\alpha,\lambda_1,m_2,\alpha,\lambda_2)$ and $X_\alpha \sim S(m_1,m_2,\alpha)$. Then
 	\begin{align}\label{stsdbd}
 		d_{K}(X_\alpha,X) \leq C_1 \lambda_1^{\alpha+\frac{1}{2}}+C_{2} \lambda_2^{\alpha+\frac{1}{2}},
 	\end{align}
 	where $ C_{1},C_{2}>0$ are independent of $\lambda_{1}$ and $\lambda_{2}$.
 \end{thm}
 \begin{proof}
 	For $h\in\mathcal{H}_K$, from \eqref{PP2:e15}, we get 
 	\begin{align}
 	   \mathbb{E}[h(X_\alpha)]- \mathbb{E}[h(X)]&=\mathbb{E}[\mathcal{A} f (X_\alpha)]=\mathbb{E}\left[\mathcal{A} f (X_\alpha)-\mathcal{A}_\alpha f (X_\alpha)  \right],\label{stsdbd0}
 	\end{align}
 	\noindent
 	since $$\mathbb{E}(\mathcal{A}_\alpha f(X_\alpha))=\mathbb{E}\left(-X_\alpha f(X_\alpha)+ \int_{\mathbb{R}}uf(X_\alpha+u)\nu_{\alpha}(du)\right)=0,~f\in \mathcal{S} (\mathbb{{R}}),$$
 	\noindent
 	 where $\nu_{\alpha}$ is the L\'evy measure given in \eqref{levyst0} (see \eqref{StenIdstable}). 
 	 
 	 \noindent
 	 Then, from \eqref{stsdbd0}, we have
 	\begin{align}
 		\nonumber\bigg|\mathbb{E}[h(X_\alpha)]- \mathbb{E}[h(X)]   \bigg|&=\bigg|\mathbb{E}\bigg[\bigg(-X_\alpha f(X_\alpha) + \displaystyle\int_{\mathbb{R}} u f(X_\alpha + u) \nu_{ts} (du)  \bigg)\\
 	\nonumber	&\quad\quad\quad\quad - \bigg(-X_\alpha f(X_\alpha) + \displaystyle\int_{\mathbb{R}} u f(X_\alpha + u) \nu_{\alpha} (du)  \bigg)  \bigg] \bigg|\\
 		\nonumber&=\bigg|\mathbb{E}\bigg[ \displaystyle\int_{\mathbb{R}} u f(X_\alpha + u) \nu_{ts} (du) - \displaystyle\int_{\mathbb{R}} u f(X_\alpha + u) \nu_{\alpha} (du)   \bigg] \bigg|\\
 		\nonumber&=\bigg|\mathbb{E} \bigg[\bigg(m_1 \displaystyle\int_{0}^{\infty}uf(X_\alpha +u) \frac{e^{-\lambda_1 u}}{u^{1+\alpha}}du \\
 		\nonumber& \quad\quad\quad+  m_2 \displaystyle\int_{-\infty}^{0}uf(X_\alpha +u) \frac{e^{-\lambda_2 |u|}}{|u|^{1+\alpha}}du   \bigg)\\
 		\nonumber& \quad\quad\quad\quad\quad\quad- \bigg( m_1 \displaystyle\int_{0}^{\infty}uf(X_\alpha +u) \frac{du}{u^{1+\alpha}} \\
 		\nonumber& \quad\quad\quad\quad\quad\quad\quad+  m_2 \displaystyle\int_{-\infty}^{0}uf(X_\alpha +u) \frac{du}{|u|^{1+\alpha}}    \bigg)    \bigg]\bigg|\\
 		\nonumber&=\bigg|\mathbb{E} \bigg[m_1\displaystyle\int_{0}^{\infty} \frac{(e^{-\lambda_1 u}-1)}{u^{1+\alpha}}uf(X_\alpha +u)du\\
 		&\quad\quad-m_2\displaystyle\int_{0}^{\infty} \frac{(e^{-\lambda_2 u}-1)}{u^{1+\alpha}}uf(X_\alpha -u)du   \bigg]   \bigg|.\label{stsdbd1}
 		\end{align}
 		\noindent
 		Now applying triangle and Cauchy-Schwartz inequalities in \eqref{stsdbd1}, we get
 		\begin{align}
 		\nonumber d_{K}(X_\alpha,X)&\leq m_1\bigg\{ \displaystyle\int_{0}^{\infty}\bigg( \frac{(e^{-\lambda_1 u}-1)}{u^{1+\alpha}}\bigg)^{2} du \bigg\}^{\frac{1}{2}}\mathbb{E}\bigg(\int_{0}^{\infty}u^{2}f^{2}(X_\alpha +u) du \bigg)^{\frac{1}{2}}\\
 			\nonumber&\quad\quad+ m_2\bigg\{ \displaystyle\int_{0}^{\infty}\bigg( \frac{(e^{-\lambda_2 u}-1)}{u^{1+\alpha}}\bigg)^{2} du \bigg\}^{\frac{1}{2}}\mathbb{E}\bigg(\int_{0}^{\infty}u^{2}f^{2}(X_\alpha -u) du \bigg)^{\frac{1}{2}}\\
 	\nonumber	&=\lambda_{1}^{\alpha +\frac{1}{2}}m_1M^{\frac{1}{2}}(\alpha)\mathbb{E}\bigg(\int_{0}^{\infty}u^{2}f^{2}(X_\alpha +u) du \bigg)^{\frac{1}{2}}\\
 		&\quad\quad+\lambda_{2}^{\alpha +\frac{1}{2}}m_2M^{\frac{1}{2}}(\alpha)\mathbb{E}\bigg(\int_{0}^{\infty}u^{2}f^{2}(X_\alpha -u) du  \bigg)^{\frac{1}{2}},\label{stsdbd2}
 	\end{align}
 \noindent
 where $M(\alpha)=\displaystyle\int_{0}^{\infty}\bigg( \frac{(e^{- u}-1)}{u^{1+\alpha}}\bigg)^{2} du <\infty$ (see \cite[p.169]{kol1}). Also, $\mathbb{E}(\int_{0}^{\infty}u^{2}f^{2}(X_\alpha +u) du  )^{\frac{1}{2}}$ and $\mathbb{E}(\int_{0}^{\infty}u^{2}f^{2}(X_\alpha -u) du  )^{\frac{1}{2}}$ are finite, since $f\in \mathcal{S}(\mathbb{{R}}).$ Now setting
 \begin{align*}
 C_1&=m_1M^{\frac{1}{2}}(\alpha)\mathbb{E}\bigg(\int_{0}^{\infty}u^{2}f^{2}(X_\alpha +u) du \bigg)^{\frac{1}{2}}<\infty,\\
 \text{and }C_2&=m_2M^{\frac{1}{2}}(\alpha)\mathbb{E}\bigg(\int_{0}^{\infty}u^{2}f^{2}(X_\alpha -u) du  \bigg)^{\frac{1}{2}}<\infty,
 \end{align*}
 \noindent
 in \eqref{stsdbd2}, we get
 	\begin{align*}
 		d_{K}(X_\alpha,X) \leq C_1\lambda_{1}^{\alpha +\frac{1}{2}} +C_2\lambda_{2}^{\alpha +\frac{1}{2}},
 	\end{align*}
 	\noindent
 where $ C_{1},C_{2}>0$ are independent of $\lambda_{1}$ and $\lambda_{2}$. This proves the result.	
 \end{proof}

\noindent
Next, we state a result that gives the limiting distribution of a sequence of tempered stable random variables.
 \begin{lem}(\cite[Proposition 3.1]{k13})\label{tsa0}
 	Let $m_{1},m_{2},m_{i,n},\lambda_{i,n} \in (0,\infty)$ and $\alpha_{1},\alpha_{2},\alpha_{i,n} \in [0,1)$, for $i=1,2$. Also, let $X_n \sim \text{TSD}(m_{1,n},\alpha_{1,n},\lambda_{1,n},m_{2,n},\alpha_{2,n},\lambda_{2,n})$ and $X \sim \text{TSD}(m_{1},\alpha_{1},\lambda_{1},m_{2},\alpha_{2},\lambda_{2} )$. If 
 		$(m_{1,n},\alpha_{1,n},\lambda_{1,n},m_{2,n},\alpha_{2,n},\lambda_{2,n} ) \to (m_{1},\alpha_{1},\lambda_{1},m_{2},$  $\alpha_{2},\lambda_{2}) \text{ as } n\to \infty,  $
 		then $X_n \overset{L}{\to} X.$ 
 \end{lem}
  \noindent
 The following theorem gives the error in the approximation of $X_n$ to $X$. 
\begin{thm}\label{bd:th2}
	Let $X_n$ and $X$ be defined as in Lemma \ref{tsa0}. Then
	\begin{align}
		\nonumber d_{\mathcal{H}_3} (X_n,X) &\leq  \left| C_1(X_n) -C_1(X)  \right|+ \frac{1}{2} \left|C_2(X_n) -C_2(X)   \right|\\
			&\quad\quad+\frac{1}{6} C_2(X)\left| \frac{|C_3(X_n)|}{C_2(X_n)}- \frac{|C_3(X)|}{C_2(X)}\right|,\label{bdtsa}
	\end{align}
	\noindent
	where $C_j(X),$ $j=1,2,3$, denotes the $j$-th cumulant of $X$ and $d_{\mathcal{H}_3}$ is defined in \eqref{SWD00}.
\end{thm}
\begin{proof}
	Let $h\in\mathcal{H}_3$ and $f$ be the solution to the Stein equation \eqref{STOPTSD}. Then
	\begin{align}
	\nonumber	\mathbb{E}[h(X_n)] - \mathbb{E}[h(X)] &= \mathbb{E}[\mathcal{A} f(X_n)]\\
	\nonumber	&=\mathbb{E} \bigg[ -X_n f(X_n) + \displaystyle\int_{\mathbb{R}} u f(X_n +u) \nu_{ts} (du) \bigg]\\
		&=\mathbb{E} \bigg[\bigg(-X_n +C_1 (X)\bigg)f(X_n) + C_2(X) f^{\prime} (X_n+ Y) \bigg],\label{bdtsa00}
	\end{align}
	\noindent
	where the last equality follows by \eqref{eqn1}, and $Y$ has the density given in \eqref{pdfzbias}. 
	
	\vskip 2ex
	\noindent
	Since $X_n \sim \text{TSD}(m_{1,n},\alpha_{1,n},\lambda_{1,n},m_{2,n},\alpha_{2,n},\lambda_{2,n}),$ then by Proposition \ref{th1}, we have
	\begin{align}\label{bdtsa01}
	\mathbb{E} \bigg[ -X_n f(X_n) + \displaystyle\int_{\mathbb{R}} u f(X_n +u) \nu_{ts}^n (du) \bigg]=0,
	\end{align}
	where $\nu_{ts}^n$ is the L\'evy measure, given by
	$$\nu_{ts}^n(du)=\left(\frac{m_{1,n}  }{u^{1+\alpha_{1,n}}}e^{-\lambda _{1,n} u}\mathbf{I}_{(0,\infty)}(u)+\frac{m_{2,n} }{|u|^{1+\alpha_{2,n}}}e^{-\lambda_{2,n}|u|}\mathbf{I}_{(-\infty,0)}(u)\right)du.$$
	\noindent
	Also, by Lemma \ref{theorem1}, the identity in \eqref{bdtsa01} can be seen as
	\begin{align}\label{bdtsa02}
		\mathbb{E} \bigg[\bigg(-X_n +C_1 (X_n)\bigg)f(X_n) + C_2(X_n) f^{\prime} (X_n+ Y_n) \bigg]=0,
	\end{align}
	where $Y_n$ has the density
	\begin{align}\label{bdtsa03}
		f_{n}(u)=\frac{[\int_{u}^{\infty}y \nu_{ts}^n(dy)] \textbf{I}_{(0,\infty)}(u)-[\int_{-\infty}^{u}y \nu_{ts}^n(dy)]\textbf{I}_{(-\infty,0)}(u)}{C_2 (X_n)},~u\in \mathbb{{R}}.
	\end{align}
	\noindent
	 Using \eqref{bdtsa02} in \eqref{bdtsa00}, we get
	\begin{align}
	\nonumber	\bigg|\mathbb{E}[h(X_n)] - \mathbb{E}[h(X)]\bigg| &=\bigg|\mathbb{E} \bigg[\bigg( (-X_n +C_1 (X))f(X_n) + C_2(X) f^{\prime} (X_n+ Y)\bigg)\\
	\nonumber	&\quad\quad - \bigg((-X_n +C_1 (X_n))f(X_n) + C_2(X_n) f^{\prime} (X_n+ Y_n) \bigg) \bigg]\bigg|\\
	\nonumber& \leq \left| C_1(X_n)-C_1(X) \right|\|f\|\\
	\nonumber& \quad\quad\quad + \mathbb{E}\left|C_2 (X_n)f^{\prime} (X_n+ Y_n) -C_2(X)f^{\prime}(X_n+Y) \right|\\
\nonumber	& \leq \left| C_1(X_n)-C_1(X) \right|\|f\|\\
	\nonumber& \quad\quad\quad + \mathbb{E} \bigg|(C_2(X_n)- C_2 (X)) f^{\prime}(X_n + Y_n)  \bigg|\\
\nonumber	&\quad\quad\quad \quad\quad+ C_2(X) \mathbb{E} \bigg| f^{\prime} (X_n+Y_n)- f^{\prime} (X_n+Y) \bigg|\\
\nonumber& \leq \|h^{(1)}\| |C_1 (X_n) - C_1 (X)| + \frac{\|h^{(2)}\|}{2} |C_2 (X_n)  -C_2 (X)|\\
&\quad\quad\quad+ C_2 (X) \frac{\|h^{(3)}\|}{3} \bigg|\mathbb{E}|Y_n|-\mathbb{E}| Y|\bigg|,\label{bdtsa04}
	\end{align}
\noindent
where the last inequality follows by applying the estimates given in Lemma \ref{th3}. From \eqref{bdtsa03} and \eqref{pdfzbias}, it can be verified that (see \eqref{mom0})
\begin{align}\label{bdtsa05}
	\mathbb{E}|Y_n|= \frac{|C_{3} (X_n)|}{2C_2 (X_n)} \text{ and }\mathbb{E}|Y|= \frac{|C_{3} (X)|}{2C_2 (X)}.
\end{align}
\noindent
Using \eqref{bdtsa05} in \eqref{bdtsa04}, we get our desired result.		
\end{proof}

\begin{rem}
		(i) Note that if 
	$(m_{1,n},\alpha_{1,n},\lambda_{1,n},m_{2,n},\alpha_{2,n},\lambda_{2,n} ) \to (m_{1},\alpha_{1},\lambda_{1},m_{2},$  $\alpha_{2},\lambda_{2}) \text{ as } n\to \infty,  $ 
	then $C_j(X_n) \to C_j(X),$ $j=1,2,3$, and $d_{\mathcal{H}_3}(X_n,X)=0$, as expected.
	
	\vskip 1ex
	\item [(ii)] Note also that if $m_{1,n}=m_{2,n}, \alpha_{1,n}=\alpha_{2,n},\lambda_{1,n}=\lambda_{2,n},m_1=m_2, \alpha_1=\alpha_2,$ and $\lambda_1=\lambda_2$, then $C_j(X_n)=C_j(X)=0,$ $j=1,3$. Under these conditions, from \eqref{bdtsa}, we get
	\begin{align*}
		d_{\mathcal{H}_3} (X_n,X) &\leq \frac{1}{2}|C_2(X_n)- C_2(X)|\\
		&= \bigg|\Gamma (2- \alpha_{1,n}) \frac{m_{1,n}}{\lambda_1^{2-\alpha_{1,n}}}-\Gamma (2- \alpha_1) \frac{m_1}{\lambda_1^{2-\alpha_1}}  \bigg|.
	\end{align*}
	If in addition $C_2(X_n) \to C_2(X)$, $X_n \overset{L}{\to} X$, as $n\to \infty$.		
\end{rem}
\noindent
Next, we discuss two examples. Our first example yields the error in approximating a TSD by a normal distribution.
\begin{exmp}[Normal approximation to a TSD]
	Let $X_n \sim \text{TSD}(m_{1,n},\alpha_{1,n},\lambda_{1,n},$ $m_{2,n},\alpha_{2,n},\lambda_{2,n})$, $X_m\sim \text{SVGD}(m,\sqrt{2m}/\lambda)$ and $X_\lambda \sim \mathcal{N}(0,\lambda^2) $. Recall from Section \ref{TPOTST} that,
	$\text{SVGD}(m,\sqrt{2m}/\lambda) \overset{d}{=}\text{TSD}(m,0,\sqrt{2m}/\lambda,m,0,\sqrt{2m}/\lambda)$. Then, the cf of $\text{SVGD}(m,\sqrt{2m}/\lambda)$ is
	\begin{align}
	\label{svgcf000}	\phi_{sv}(z)&= \bigg(1+ \frac{z^2 \lambda^2}{2m} \bigg)^{-m},~z\in\mathbb{{R}},\\
		&=\exp\bigg(\displaystyle\int_{\mathbb{R}} (e^{izu}-1)\nu_{sv}(du)  \bigg),
	\end{align}  
	\noindent
	where the L\'evy measure $\nu_{sv}$ is 
	\begin{align*}
		\nu_{sv}(du)=\bigg(\frac{m}{u}e^{- \frac{\sqrt{2m}}{\lambda} u}\textbf{I}_{(0,\infty)}(u)+ \frac{m}{|u|}e^{- \frac{\sqrt{2m}}{\lambda} |u|}\textbf{I}_{(-\infty,0)}(u) \bigg).
	\end{align*}
\noindent
Note from \eqref{svgcf000} that, 
\begin{align*}
	\lim_{m\to\infty}\phi_{sv}(z)=e^{-\frac{\lambda^2 z^2}{2}}.
\end{align*}	
\noindent
That is, $X_m \overset{L}{\to} X_\lambda \sim \mathcal{N}(0,\lambda^2),$ as $m\to\infty.$ Also, it follows \cite[Theorem 7.12]{vill} that, if $X_m \overset{L}{\to} X_\lambda,$ as $m \to \infty$,
\vfill
\vfill

\begin{align}\label{svgcf1100}
d_{\mathcal{H}_3}(X_n,X_\lambda)= \lim_{m \to \infty}d_{\mathcal{H}_3}(X_n,X_m).
\end{align}
\noindent
 Applying Theorem \ref{bd:th2} to $X=X_m$, and taking limit as $m \to \infty$, we get from \eqref{svgcf1100}
\begin{align}
\nonumber d_{\mathcal{H}_3}(X_n,X_\lambda)&\leq \lim_{m \to \infty}\bigg( \left| C_1(X_n) -C_1(X_m)  \right|+ \frac{1}{2} \left|C_2(X_n) -C_2(X_m)   \right|\\
\nonumber&\quad\quad \quad\quad+\frac{1}{6} C_2(X_m)\left| \frac{|C_3(X_n)|}{C_2(X_n)}- \frac{|C_3(X_m)|}{C_2(X_m)}\right| \bigg)\\
&= |C_1(X_n)| + \frac{1}{2}|C_2(X_n) -\lambda^2|+ \frac{1}{6}\lambda^2 \frac{|C_{3}(X_n)|}{C_2(X_n)},\label{natsd}
\end{align}
which gives the error in the approximation between $X_n$ and $X_\lambda$. Note that
\begin{align*}
C_1(X_n)&=\mathbb{E}(X_n)=\Gamma (1- \alpha_{1,n}) \frac{m_{1,n}}{\lambda_{1,n}^{1-\alpha_{1,n}}}-\Gamma (1- \alpha_{2,n}) \frac{m_{2,n}}{\lambda_{2,n}^{1-\alpha_{2,n}}},\\
C_2(X_n)&=Var(X_n)=\Gamma (2- \alpha_{1,n}) \frac{m_{1,n}}{\lambda_{1,n}^{2-\alpha_{1,n}}}+\Gamma (2- \alpha_{2,n})\frac{m_{2,n}}{\lambda_{2,n}^{1-\alpha_{2,n}}},\\
\text{and }C_3(X_n)&=\Gamma (3- \alpha_{1,n}) \frac{m_{1,n}}{\lambda_{1,n}^{3-\alpha_{1,n}}}-\Gamma (3- \alpha_{2,n}) \frac{m_{2,n}}{\lambda_{2,n}^{3-\alpha_{2,n}}}.
\end{align*}
\noindent
When $C_j(X_n) \to 0$, for $j=1,3$ and $C_2(X_n) \to \lambda^2$, from \eqref{natsd}, we have $X_n \overset{L}{\to} X_\lambda$, as $n \to \infty$.

\end{exmp}
\begin{exmp}[Variance-gamma approximation to a TSD]
Let $X_n \sim \text{TSD}(m_{1,n},$ $\alpha_{1,n},\lambda_{1,n},m_{2,n},\alpha_{2,n},\lambda_{2,n})$ and $X_v \sim$VGD$(m,\lambda_1,\lambda_2)$. Then
\begin{align*}
C_1(X_v)&=m \bigg(\frac{1}{\lambda_1}- \frac{1}{\lambda_2}\bigg),\\
C_2(X_v)&=m \bigg(\frac{1}{\lambda_1^2}+ \frac{1}{\lambda_2^2}\bigg),\\
\text{and }C_3(X_v)&=2m \bigg(\frac{1}{\lambda_1^3}- \frac{1}{\lambda_2^3}\bigg).
\end{align*}
\noindent
 Now applying Theorem \ref{bd:th2} to $X=X_v$, we get
	\begin{align*}
	\nonumber d_{\mathcal{H}_3} (X_n,X_v) &\leq  \left| C_1(X_n) - \frac{m(\lambda_2-\lambda_1)}{\lambda_1\lambda_2}  \right|+ \frac{1}{2} \left|C_2(X_n) - \frac{m(\lambda_1^2+\lambda_2^2)}{\lambda_1^2 \lambda_2^2}    \right|\\
	& \quad\quad\quad\quad+ \frac{1}{6}m\frac{\lambda_1^2+\lambda_2^2}{\lambda_1^2\lambda_2^2}   \left|\frac{|C_3(X_n)|}{C_2(X_n)}-\frac{2|\lambda_{2}^{3}-\lambda_{1}^{3}|}{\lambda_1 \lambda_2 (\lambda_1^2+ \lambda_2^2)}   \right|,
	\end{align*}
	 which gives the error in the approximation between $X_n$ and $X_v$. When $C_j(X_n) \to C_j(X_v)$, for $j=1,2,3,$ we have $X_n \overset{L}{\to}X_v$, as $n\to \infty$.
\end{exmp}

\setstretch{1}
 
\end{document}